\documentclass{amsart}

\pdfoutput=0
\usepackage[francais,english]{babel}

\newtheorem{theo}{Theorem}[section]
\newtheorem{lemm}[theo]{Lemma}
\newtheorem*{question}{Question}

\theoremstyle{definition}

\newtheorem*{rema}{Remark}

\usepackage{mathrsfs}
\usepackage{array}
\usepackage[all]{xy}
\usepackage{graphicx}
\usepackage{color}
\usepackage{epsfig}
\usepackage{url}

\newcommand{\van}{{\rm mult}}

\author[S. R. Kaschner]{Scott R. Kaschner}
\address{Butler University Department of Mathematics \& Actuarial Science\\
Jordan Hall, Room 270\\
4600 Sunset Ave\\
Indianapolis, IN 46208,
 United States,}
\email{skaschne@butler.edu}

\author[R. A. P\'erez]{Rodrigo A. P\'erez}
\address{IUPUI Department of Mathematical Sciences\\
LD Building, Room 270\\
402 North Blackford Street\\
Indianapolis, IN 46202-3267,
 United States.}
\email{rperez@math.iupui.edu}
\urladdr{www.math.iupui.edu/~rperez}

\author[R. K. W. Roeder]{Roland~K.~W.~Roeder}
\address{IUPUI Department of Mathematical Sciences\\
LD Building, Room 270\\
402 North Blackford Street\\
Indianapolis, IN 46202-3267,
 United States.}
\email{rroeder@math.iupui.edu}
\urladdr{www.math.iupui.edu/~rroeder}

\title[Rational maps of $\mathbb{CP}^2$ with $\lambda_1=\lambda_2$ and no invariant foliation]{Examples of rational maps of $\mathbb{CP}^2$ with equal dynamical degrees and no~invariant~foliation}

\begin{document}

\begin{abstract}
We present simple examples of rational maps of the complex projective plane with equal first and second dynamical degrees and no invariant foliation.
\end{abstract}

\maketitle

\section{Introduction}

A meromorphic map $\varphi: X \dashrightarrow X$ of a compact K\"ahler manifold $X$ induces a well-defined pullback action $\varphi^*: H^{p,p}(X) \rightarrow H^{p,p}(X)$ for each $1 \leq p \leq \dim(X)$.   The $p$-th dynamical degree
\begin{eqnarray}\label{EQN:DEF_DYN_DEG}
\lambda_p(\varphi) := \lim_{n \rightarrow \infty} \|(\varphi^n)^*:H^{p,p}(X) \rightarrow H^{p,p}(X)\|^{1/n}
\end{eqnarray}
describes the asymptotic growth rate of the action of iterates of $\varphi$ on $H^{p,p}(X)$.   Originally, the dynamical degrees were introduced by Friedland \cite{FRIEDLAND} and later by Russakovskii and Shiffman
\cite{RUSS_SHIFF} and shown to be invariant under birational conjugacy by Dinh and Sibony \cite{DS_BOUND}.  Note that dynamical degrees were originally defined with the limit in (\ref{EQN:DEF_DYN_DEG})
replaced by ${\rm lim sup}$. However, it was shown in \cite{DS_BOUND,DS_REGULARIZATION} that the limit always exists.

One says that $\varphi$ is {\em cohomologically hyperbolic} if one of the dynamical degrees is strictly larger than all of the others.  
In this case, there is a conjecture \cite{GUEDJ_ERGODIC,GUEDJ_HABILITATION} that describes the ergodic properties of $\varphi$.  This conjecture has been proved in several particular sub-cases \cite{GUEDJ,DDG,DNT_ENTROPY}. 

What happens in the non-cohomologically hyperbolic case?  If a meromorphic map preserves a fibration, then there are nice formulae relating the dynamical degrees of the map \cite{DN,DNT}. This is the case in the following two examples:
{\renewcommand{\labelenumi}{\alph{enumi})}
\begin{enumerate}
  \item It was shown in \cite{DF} that bimeromorphic  maps of
surfaces that are not cohomologically hyperbolic ($\lambda_1(\varphi) = \lambda_2(\varphi)
= 1$) always preserve an invariant fibration.
  \item Meromorphic maps that are not cohomologically hyperbolic arise naturally
when studying the spectral theory of operators on self-similar spaces
\cite{SABOT,BG,GZ,BGN}.   All the examples studied in that context preserve an invariant fibration.  In several cases, this fibration made it significantly easier to compute the limiting spectrum of the action.
\end{enumerate} }

Based on this evidence, Guedj asked in \cite[p. 103]{GUEDJ_HABILITATION} whether every non-cohomologically hyperbolic map preserves a fibration.

We will prove:
\begin{theo}\label{THM:MAIN}
The rational map $\varphi: \mathbb{CP}^2 \dashrightarrow \mathbb{CP}^2$  given by
\begin{equation}
\varphi[X:Y:Z]\ =\ [-Y^2:X(X-Z):-(X+Z)(X-Z)]
\end{equation}
is not cohomologically hyperbolic ($\lambda_1(\varphi) = \lambda_2(\varphi)
= 2$), and no iterate of $\varphi$ preserves a singular holomorphic foliation. Moreover, for a Baire generic set of automorphisms $A \in {\rm PGL}(3,\mathbb{C})$, the composition $A \circ {\varphi^4}$ has the same properties.
\end{theo}
\noindent
Since preservation of a fibration is a stronger condition than preservation of a singular foliation, $\varphi$ provides an answer to the  question posed by Guedj.

After reading a preliminary version of this paper, Charles Favre asked if the
same behavior can be found for a polynomial map of $\mathbb{C}^2$.  In \cite[\S
7.2]{FAVRE_JONSSON}, there is a list of seven types of non-cohomologically
hyperbolic polynomial mappings. Our method of proof does not apply in most of these examples for trivial reasons, but we found a map of type (3) for which the same result holds:
\begin{theo}\label{THM2}
The polynomial map 
\begin{equation}
\psi(x,y) := \big(x(x-y)+2,(x+y)(x-y)+1\big)
\end{equation}
extends as a rational map $\psi: \mathbb{CP}^2 \dashrightarrow \mathbb{CP}^2$ that is not cohomologically hyperbolic ($\lambda_1(\psi) = \lambda_2(\psi)= 2$), and no iterate of $\psi$ preserves a singular holomorphic foliation on $\mathbb{CP}^2$.
\end{theo}

Non-cohomologically hyperbolic maps of $3$-dimensional manifolds $X$ arise naturally as certain pseudo-automorphisms that are ``reversible'' on
$H^{1,1}(X)$  \cite{BK2,BK3,PZ}.  For these mappings it follows from Poincar\'e duality that $\lambda_1(\varphi) = \lambda_2(\varphi)$.  Recently, Bedford, Cantat, and Kim \cite{BCK} have found a reversible pseudo-automorphism of an iterated blow-up of $\mathbb{CP}^3$ which does not preserve any invariant foliation.  It also answers the question posed by Guedj.

Many authors have studied meromorphic (and rational) maps that preserve foliations and algebraic webs, including \cite{FP,CF,PS,DJ1,DJ2,Br,DF}.  Since the proof of Theorem \ref{THM:MAIN} is self-contained and provides insight on the mechanism that prevents $\varphi$ from preserving a foliation, we will provide a direct proof rather than appealing to results from these previous papers.

Let us give a brief idea of the proof of Theorem~\ref{THM:MAIN}.  The fourth iterate ${\varphi^4}$ has an indeterminate point $p$ that is blown-up by ${\varphi^4}$ to a singular curve $C$.  Any foliation $\mathcal{F}$ must be either generically transverse to $C$ or have $C$ as a leaf.  This allows us to show that $({\varphi^4})^{\ast}\mathcal{F}$ must be singular at either at $p$ or at some preimage $r$ of the singular point.  Both of these points have infinite ${\varphi^4}$ pre-orbits, at each point of which $\varphi^{4n}$ is a finite map. This generates a sequence of distinct points $\{a_{-n}\}_{n=0}^{\infty}$ at such that $(\varphi^{4n})^{\ast}\mathcal{F}$ is singular at $a_{-n}$.  If $(\varphi^{\ell})^{\ast}\mathcal{F} = \mathcal{F}$ for some $\ell$ this implies that $\mathcal{F}$ is singular at infinitely many points, providing a contradiction.  The same method of proof applies for the polynomial map in Theorem~\ref{THM2}.

\begin{question}
Our proof of Theorem~\ref{THM2} relies on the compactness of $\mathbb{CP}^2$.  Does $\psi$ preserve a foliation on $\mathbb{C}^2$?
\end{question}

If the foliation is algebraic, then it extends to a foliation of $\mathbb{CP}^2$, and the answer follows from Theorem \ref{THM2}.  However, if the foliation does not extend because it has an infinite set of singularities in $\mathbb C^2$ that accumulate to the line at infinity, then we do not yet know how to address the situation.

It would also be interesting to place the maps $\varphi, \psi$ in the context of the conjecture from \cite{GUEDJ_ERGODIC}:

\begin{question}
Do these maps have topological entropy $\log2$?  Can one find a (unique) measure of maximal entropy? As $n \rightarrow \infty$, what is the asymptotic behavior of the periodic points of period $n$, and are the periodic points predominantly saddle-type, repelling, or degenerate?
\end{question}

In \S \ref{SEC:BACKGROUND} we provide background on the transformation of
foliations and fibrations by rational maps.  In \S \ref{SEC:STRUCTURE} we
describe some basic properties of $\varphi$ and we prove that
$\lambda_1(\varphi) = 2 = \lambda_2(\varphi)$, showing that $\varphi$ is not
cohomologically hyperbolic.  In \S\ref{SEC:PROOF1} we prove that no iterate of
$\varphi$ preserves a foliation, and conclude the section with a summary of the
properties of $\varphi$ that make the proof of Theorem \ref{THM:MAIN} work. In
\S\ref{SEC:PROOF2} we prove Theorem~\ref{THM2} by computing that
$\lambda_1(\psi) = 2 = \lambda_2(\psi)$ and verifying that $\psi$ has the
properties listed in \S\ref{SEC:PROOF1}.  In the (last) section \S
\ref{SEC:PROOF3} we complete the proof of Theorem \ref{THM:MAIN}.

\vspace{0.1in}

\noindent {\bf Acknowledgements.} ---
We have benefited greatly from discussions with Eric Bedford, Serge Cantat,
Laura DeMarco, Jeffrey Diller, Tien-Cuong Dinh, and Charles Favre.  We thank the referee for several helpful comments. This work
was supported in part by the National Science Foundation grants DGE-0742475 (to
S.R.K), DMS-1102597 (to R.K.W.R.), DMS-1348589 (to R.K.W.R.), and IUPUI startup
funds (to R.A.P. and R.K.W.R.).

\section{Background} \label{SEC:BACKGROUND} 

To make this paper self-contained, in this section we will present some background about singular holomorphic foliations on complex surfaces and their pullback under rational maps.  More details can be found in \cite{ABATE,Br,CERVEAU,FP, IY,REBELO}.


Recall that a {\em holomorphic foliation} $\mathcal{F}$ on a complex surface $X$ is a partition 
\begin{eqnarray*}
X = \bigsqcup_\alpha L_\alpha
\end{eqnarray*}
of $X$ into the disjoint union of connected subsets $L_\alpha$, that is locally biholomorphically equivalent to $\mathbb{D}^2=\bigsqcup_{|y| < 1} \{|x| < 1\} \times \{y\}$, the standard holomorphic foliation. 
A {\em singular holomorphic foliation} on $X$ is a holomorphic foliation $\mathcal{F}$ on $X\setminus P$, where $P$ is a discrete set of points through which $\mathcal{F}$ does not extend. 

\begin{theo}[Ilyashenko \cite{IL72} and {\cite[Thm. 2.22]{IY}}]\label{THM:EXTENSION} In a neighborhood of each singular point $p \in X$, $\mathcal{F}$ is 
generated as the integral curves of some holomorphic $1$-form.
\end{theo}

This allows for any
 singular holomorphic foliation $\mathcal{F}$ on a surface $X$ to be described by
 an open cover $\{U_i\}$ of $X$ and a system of holomorphic $1$-forms $\omega_i$ on $U_i$ with isolated zeros that satisfy the compatibility condition $\omega_i =  g_{ij} \omega_j$ for some non-vanishing holomorphic functions
\begin{equation*}
g_{ij}: U_i \cap U_j \rightarrow \mathbb{C}.
\end{equation*}
Within a given $U_i$ the tangent direction to a leaf is described by the kernel of~$\omega_i$.  The zeros
of $\omega_i$ correspond to the singular points of~$\mathcal{F}$.

Let $\varphi: X \rightarrow Y$ be a dominant holomorphic map between two
surfaces and let $\mathcal{F}$ be a singular holomorphic foliation on $Y$ given
by $\{(U_i,\omega_i)\}$.  Recall that the {\em pullback} $\varphi^*
\mathcal{F}$ is defined by $\{(\varphi^{-1}(U_i),\hat \omega_i)\}$ where each
form $\hat \omega_i$ is obtained by rescaling $\varphi^* \omega_i$, i.e.,
dividing it by a suitable holomorphic function in order to eliminate any
non-isolated zeros.

Let us now see what what happens in the case that $\varphi: X \dashrightarrow
Y$ is a dominant rational map between two complex projective algebraic
surfaces.  Consider a point $p$ of the indeterminacy set $\mathcal{I}_{\varphi}$ of
$\varphi$.  A sequence of blow-ups $\pi:~\tilde{X}~\rightarrow~X$ resolves the
indeterminacy of $\varphi$ at $p$ if $\varphi$ lifts to a rational map $\tilde{\varphi}: \tilde{X}~\dashrightarrow~Y$ which is holomorphic in an open neighborhood of
$\pi^{-1}(p)$ and makes the following diagram commute
\begin{eqnarray} \label{GENERAL_RESOLUTION}
\xymatrix{\widetilde{X} \ar[d]^\pi  \ar @{-->}[dr]^{\widetilde{\varphi}}  & \\
X \ar @{-->}[r]^\varphi & Y, }
\end{eqnarray}
where $\varphi \circ \pi$ and $\tilde{\varphi}$ are both defined.

It is a well-known fact that one can do a minimal sequence of blow-ups $\pi: \tilde{X} \rightarrow X$  over $\mathcal{I}_\varphi$ so that $\varphi$ lifts to a holomorphic map $\tilde{\varphi}: \tilde{X}~\rightarrow~Y$ resolving all of the points of $\mathcal{I}_\varphi$ (see, for example, \cite[Ch. IV, \S3.3]{SHAF}).  The pull-back of $\mathcal{F}$ under $\varphi$ is defined by $\varphi^* \mathcal{F} = \pi_* \tilde{\varphi}^* \mathcal{F}$.

{\bf Conventions.} 
For the remainder of the paper we will use ``surface'' to mean ``complex projective algebraic
surface'', and ``foliation'' to mean ``singular holomorphic foliation.''
Also, all rational maps will be dominant, meaning that the image is not contained within
a proper subvariety of the codomain.

The following facts will be used later.

\begin{rema}\label{LEM_COMPOSITION}
 Let $\varphi: X \dashrightarrow Y$ and $\psi: Y \dashrightarrow Z$ be any dominant rational maps. For any foliation $\mathcal{F}$ on $Z$ we have $(\psi \circ \varphi)^* \mathcal{F} = \varphi^* (\psi^* \mathcal{F})$.
\end{rema}

\begin{rema}\label{LEM_LEAF_OR_TRANSV} Let $\mathcal{F}$ be a foliation on a surface $X$ and suppose $C\subset X$ is an irreducible algebraic curve.  Then, either $C$ is a leaf of $\mathcal{F}$ or $C$ is transverse to $\mathcal{F}$ away from finitely many points.
\end{rema}

\begin{lemm}\label{PREIMAGE_SINGULAR} 
Let $\varphi: X \dashrightarrow Y$ be a dominant rational map.  Let $p \in X \setminus \mathcal{I}_\varphi$ and let $\mathcal{F}$ be a foliation on $Y$.  If $\varphi(p)$ is a singular point for $\mathcal{F}$ and $\varphi$ is finite at $p$, then $p$ is a singular point for $\varphi^*(\mathcal{F})$.
\end{lemm}

\begin{lemm}\label{CURVE_TRANS_FOLIATION}Let $\varphi: X \dashrightarrow Y$ be a dominant rational map, let $p \in \mathcal{I}_\varphi$, and suppose $\pi: \widetilde{X} \rightarrow X$ and $\widetilde{\varphi}: \widetilde{X}\dashrightarrow Y$ give a resolution of the indeterminacy at~$p$.
Then, if $\mathcal{F}$ is a foliation on $Y$ and one of the irreducible components of $\widetilde{\varphi}(\pi^{-1}(p))$ is generically transverse to $\mathcal{F}$, the pullback $\varphi^*\mathcal{F}$ is singular at $p$.
\end{lemm}

Recall that a {\em fibration} is a dominant rational map $\rho: X \dashrightarrow S$ to a non-singular algebraic curve $S$ with the property that each  fiber is connected.  The level curves of $\rho$ define a foliation on $X \setminus \mathcal{I}_\rho$ having finitely many singularities.  It thus extends as a singular foliation $\mathcal{F}$ on $X$ called the {\em foliation induced by $\rho$}.

\begin{lemm}\label{FIBRATION} Let $\rho: X \dashrightarrow S$ be a fibration inducing a foliation $\mathcal{F}$ on $X$.  A rational map $\varphi: X \dashrightarrow X$ preserves  $\mathcal{F}$  if and only if $\rho$ semi-conjugates $\varphi$ to a holomorphic map of $S$:
\begin{eqnarray}\label{EQN:CD}
\xymatrix{
X \ar @{-->}[r]^\varphi \ar @{-->}[d]^\rho & X \ar @{-->}[d]^\rho \\
S \ar[r]^\eta & S.
}
\end{eqnarray}
\end{lemm}

\begin{rema}
Using Lemma \ref{FIBRATION}, it follows immediately from Theorem \ref{THM:MAIN} that no iterate of $\varphi$ preserves a fibration.
\end{rema}

Since Lemma \ref{FIBRATION} is well known, we omit the proof.  However, we include proofs of Lemmas \ref{PREIMAGE_SINGULAR} and \ref{CURVE_TRANS_FOLIATION} below.

\begin{proof}[Proof of Lemma \ref{PREIMAGE_SINGULAR}]
Suppose that $\varphi^* \mathcal{F}$ is non-singular at $p$.  Then, we can choose local coordinates $(x_1,x_2)$ centered at $p$ so that $\varphi^* \mathcal{F}$ is given by $\hat \omega = dx_1$.

Let $(y_1,y_2)$ be local coordinates centered at $\varphi(p)$.  In these coordinates, we have $\varphi: (\mathbb{C},(0,0)) \rightarrow (\mathbb{C},(0,0))$ where $\varphi(x_1,x_2) := (\varphi_1(x_1,x_2),\varphi_2(x_1,x_2))$.
Suppose $\mathcal F$ is given by
\begin{eqnarray*}
\omega = a(y_1,y_2)dy_1 + b(y_1,y_2) dy_2
\end{eqnarray*}
with $a$ and $b$ having an isolated common zero at $(0,0)$. 

Since $\hat \omega = dx_1$, the $dx_2$ coefficient of $\varphi^* \omega$ must vanish identically:
\begin{eqnarray}
(a\circ \varphi)(x_1,x_2) \frac{\partial \varphi_1}{\partial x_2} + (b\circ \varphi)(x_1,x_2) \frac{\partial \varphi_2}{\partial x_2} \equiv 0.
\end{eqnarray}
We will show that this is impossible.

Since $\varphi$ is dominant, neither $(a\circ \varphi)$ nor $(b\circ \varphi)$ can be identically zero.  Thus,
$\frac{\partial \varphi_1}{\partial x_2} \equiv 0$ if and only if $\frac{\partial
\varphi_2}{\partial x_2} \equiv 0$, and hence, neither can vanish identically
since $\varphi$ is dominant.  Moreover, since $\varphi$ is finite, $(a\circ \varphi)$
and $(b\circ \varphi)$ have an isolated common zero at $(x_1,x_2) = (0,0)$.  This implies that $\frac{\partial \varphi_1}{\partial x_2}$ vanishes along the same curves as $b\circ \varphi$ with at least the same multiplicity (and similarly for $\frac{\partial \varphi_2}{\partial x_2}$ along the same curves as $a\circ \varphi$).  Therefore, if $\van_q(f)$ denotes the order of vanishing of a holomorphic function $f$ at point $q$, we have:
\begin{align}\label{VANISHING1}
\van_{(0,0)}\left(\frac{\partial \varphi_1}{\partial x_2}\right)\geq\van_{(0,0)}\left(b \circ \varphi\right)\geq\min\left\{\van_{(0,0)}(\varphi_1),\van_{(0,0)}(\varphi_2)\right\},
\end{align}
and a similar inequality holds for $\van_{(0,0)}\left(\frac{\partial \varphi_2}{\partial x_2}\right)$.

Note that $\varphi_1$ and $\varphi_2$ do not both vanish on $\{x_1=0\}$ because $\varphi$ does not collapse curves.  
Therefore, at least one of $\varphi_1$ and $\varphi_2$ has a monomial consisting of a positive power of $x_2$ in its power series.
Let $\ell \geq 1$ be the minimal exponent of an $x_2^\ell$ monomial in either series, and without loss of generality assume it
is present in the power series of $\varphi_1$.

Let $k > \ell$ and consider $\psi\colon(\mathbb C^2,(0,0)) \rightarrow (\mathbb C^2,(0,0))$ given by
\begin{equation*}
\psi(x_1,x_2):=\varphi(x_1^k,x_2).
\end{equation*}
Since the vertical foliation given by $dx_1 = 0$ is preserved under $(x_1,x_2)~\mapsto~(x_1^k,x_2)$, we also have that $\psi^* \mathcal{F}$ is the vertical foliation. Moreover, since $\psi$ is a composition of finite, dominant maps, (\ref{VANISHING1}) holds for $\psi$.  On the other hand, $k$ was chosen so that
\begin{eqnarray*}
\min\left\{\van_{(0,0)}(\psi_1),\van_{(0,0)}(\psi_2)\right\} = \van_{(0,0)}(\psi_1) = \ell,
\end{eqnarray*}
with $x_2^\ell$ the only term of degree $\ell$ in the expansion for $\psi_1$.  In particular,
\begin{equation*}
\van_{(0,0)}\left(\dfrac{\partial \psi_1}{\partial x_2}\right) = \ell-1 < \ell =\van_{(0,0)}(\psi_1),
\end{equation*}
which contradicts that (\ref{VANISHING1}) applies to $\psi$.
\end{proof}

\begin{proof}[Proof of Lemma \ref{CURVE_TRANS_FOLIATION}]
Let $C$ be an irreducible component of $\widetilde{\varphi}(\pi^{-1}(p))$ that is generically transverse to $\mathcal{F}$ and suppose $\mathcal{F}$ is represented by a compatible system of $1$-forms $\{(U_i,\omega_i)\}$. 
Then, the tangent spaces to generic points of $C$ are not in the kernels of any of the $1$-forms based at these points.  

 Let $E \subset \pi^{-1}(p)$ be an irreducible component that is mapped by $\widetilde{\varphi}$ onto $C$.  Then, at generic $x \in E$ we have that $x$ and $\widetilde{\varphi}(x)$ are smooth points of $E$ and $C$, respectively, and $D\widetilde{\varphi}$ is an isomorphism between their tangent spaces. This implies that $\{(\widetilde{\varphi}^{-1}(U_i),\widetilde{\varphi}^* \omega_i)\}$ are non-vanishing at generic points of $E$ and have kernel transverse to $E$ at these points.  Thus, $\widetilde{\varphi}^* \mathcal{F}$ is generically transverse to~$E$.

Now we can choose two disjoint holomorphic disks $\gamma_1$ and $\gamma_2$ within leaves of $\widetilde{\varphi}^* \mathcal{F}$ that intersect $\pi^{-1}(p)$ transversally at two distinct points of $x_1, x_2 \in E$.  Then, $\pi(\gamma_1)$ and $\pi(\gamma_2)$ are distinct integral curves of $\varphi^* \mathcal{F} = \pi_* \widetilde{\varphi}^* \mathcal{F}$ going through $p$.  This implies that $p$ is a singular point for $\varphi^* \mathcal{F}$.
\end{proof}

\section{Structure of $\varphi$}\label{SEC:STRUCTURE}

Recall that $\varphi : \mathbb{CP}^2 \dashrightarrow \mathbb{CP}^2$ is given by
\begin{equation*}
\varphi[X:Y:Z]\ =\ [-Y^2:X(X-Z):-(X+Z)(X-Z)].
\end{equation*}

\subsection{Indeterminate and Postcritical Sets}
Solving for the points where all three homogeneous coordinates of $\varphi$ are zero, we find that the only indeterminate point of $\varphi$ is
\begin{equation}
p:=[1:0:1].
\end{equation}

We have $|D\varphi|=-4Y(X-Z)^2$, so the curves
\begin{eqnarray}
L_{\rm coll}&:=&\{X=Z\}\mbox{ and }\\
L_Y&:=&\{Y=0\}
\end{eqnarray}
are critical.
Let $L_X:=\{X=0\}$, and observe that
\begin{equation}\label{FOURCYCLE}
\xymatrix{
L_{\rm coll} \setminus \{p\}  \ar[r] & [1:0:0] \ar[r] & [0:1:-1]  \ar[r] & [-1:0:1]  \ar[r] & [0:1:0] \ar@(dl,dr)[lll]
}
\end{equation}\\
We will refer to the points in this four cycle as $q_1,\ldots,q_4$, respectively. Meanwhile, 
\begin{equation*}
\xymatrix{
L_Y  \ar[r] &   L_X \ar@(dr,dl)[l],
}
\end{equation*}\\
as illustrated in Figure \ref{LINES}.

\begin{figure}
\centering
\scalebox{.7}{
\input{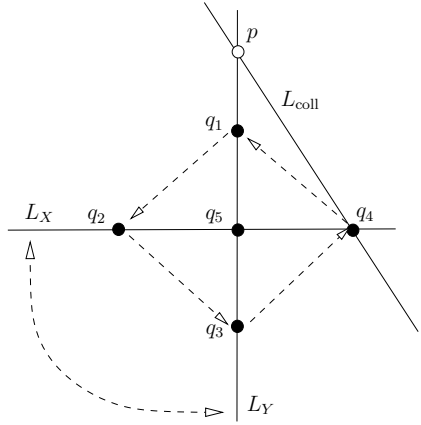}
}
\caption{\label{LINES}Postcritical set for $\varphi$}
\end{figure}

\subsection{Resonant Dynamical Degrees}
Note that the orbit of the collapsed curve $L_{\rm coll}$ lies in the four cycle (\ref{FOURCYCLE}), which is disjoint from $p$. It follows that there is no curve $V$ such that $\varphi(V)\subset\mathcal I_\varphi$, so $\varphi$ is algebraically stable by \cite[Prop. 1.4.3]{SIBONY}. Therefore, $\lambda_1(\varphi) = \deg_{\rm alg}(\varphi)=2$.  

We will show in the proof of Lemma \ref{C_SINGULARITY} that $p$ can be resolved by two blow-ups with the image of the exceptional divisors being the line $\{Y=-2Z\}$. In particular, $[1:1:-1]$ is neither a critical value nor in the image of the indeterminacy, so it is a generic point for $\varphi$.
Its two preimages under $\varphi$ are $[1:\pm i:0]$, so $\lambda_2(\varphi) = 2$. This establishes the following resonance of dynamical degrees:
\begin{lemm}\label{LEM:RESONANT_DEGREES}
The dynamical degrees coincide: $\lambda_1(\varphi)=\lambda_2(\varphi) = 2$ so that $\varphi$ is not cohomologically hyperbolic.
\end{lemm}

\subsection{Fatou Set}
Note also that $\varphi^2$ fixes both $L_X$ and $L_Y$, with each line transversally superattracting under $\varphi^2$.  In the local coordinate $z=Z/Y$, $\varphi^2_{\mid L_X}$ is $z\mapsto z^2-1$, the well-known Basilica map, with period two superattracting cycle $q_2 \leftrightarrow q_4$.  Similarly, in a suitable local coordinate, $\varphi^2_{\mid L_Y}$ is also the Basilica map, with period two superattracting cycle $q_1 \leftrightarrow q_3$. Finally, $\varphi$ (and hence $\varphi^2$) has $q_5 = [0:0:1]$ as a superattracting fixed point. We obtain five superattracting fixed points for ${\varphi^4}$. 

\begin{question}
Is the Fatou set of $\varphi$ the union of the basin of the superattracting 4-cycle $q_1,\dots,q_4$ and the basin of the superattracting fixed point $q_5$?
\end{question}

\subsection{Resolving the indeterminate point $p$.}
Proof of Theorem \ref{THM:MAIN} relies on careful analysis of ${\varphi^4}$. The indeterminacy set of ${\varphi^4}$ is
\begin{multline*}
  \mathcal{I}_{{{\varphi^4}}} = \{p\} \cup \varphi^{-1}(\{p\}) \cup \varphi^{-2}(\{p\}) \cup \varphi^{-3}(\{p\}) = \\
  \{[1:0:1]\} \cup \{[0:\pm i: 1]\} \cup \{[1:0:-1\pm i]\} \cup \{[0:i:\pm \sqrt{1 \pm i}]\}.
\end{multline*}
In the proof of Theorem \ref{THM:MAIN} we plan to apply
Lemma \ref{CURVE_TRANS_FOLIATION} to $p$, so we only resolve 
the indeterminacy of ${\varphi^4}$ at $p$.

Let $\widehat{\mathbb{CP}^2}$ be the blow-up $\mathbb{CP}^2$ at $p$ with
exceptional divisor $E_1$, and let $\widehat{L}_{\rm coll}$ denote the proper transform of $L_{\rm coll}$.  Then let $\widetilde{\mathbb{CP}^2}$ be the blow-up of $\widehat{\mathbb{CP}^2}$ at the point $E_1 \cap \widehat{L}_{\rm coll}$, resulting in a new exceptional divisor~$E_2$. We will abuse notation by denoting the proper transform of $E_1$ in $\widetilde{\mathbb{CP}^2}$ also by  $E_1$.

\begin{lemm}\label{C_SINGULARITY}
 The map ${\varphi^4}:  \mathbb{CP}^2 \dashrightarrow \mathbb{CP}^2$ lifts to a rational map $\widetilde{{\varphi^4}} : \widetilde{\mathbb{CP}^2} \dashrightarrow\mathbb{CP}^2$ that resolves the indeterminacy of ${\varphi^4}$ at $p$.  We have that $\widetilde{{\varphi^4}}(E_1) = [1:0:-9] =: s$ and that $\widetilde{{\varphi^4}}(E_2)$ is an irreducible algebraic curve $C_4$ of degree~$8$ that is singular at $s$.
\end{lemm}

\begin{proof}
We will first show that $\varphi$ lifts to a rational map $\widetilde{\varphi}:
\widetilde{\mathbb{CP}^2} \dashrightarrow \mathbb{CP}^2$ that resolves the
indeterminacy of $\varphi$ at~$p$, satisfying $\widetilde{\varphi}(E_1) = [0:1:-2]$ and $\widetilde{\varphi}(E_2) = \{Z=-2Y\}$.

We will do calculations in several systems of local coordinates in a neighborhood of $E_1$ and $E_2$.  A summary is shown in 
Figure \ref{BLOWUP}.
\begin{figure}
\scalebox{1.3}{
\epsfig{file=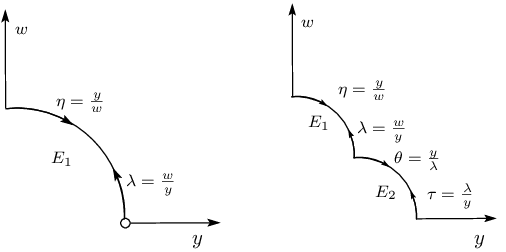}%
}
\caption{\label{BLOWUP}Left: coordinates used for the first blow-up of $p$.  The new indeterminate point $(y,\lambda) = (0,0)$ for $\widehat{\varphi}$ is marked by an open circle.  Right: coordinates used for the second blow-up.}
\end{figure}
Consider the two systems of affine coordinates on $\mathbb{CP}^2$ given by $(y,z)=\left(Y/X,Z/X\right)$ and $(x,\zeta)=(X/Y,Z/Y)$.  Let $w:=z-1$, so $(y,w)$ are coordinates that place $p$ locally at the origin.  
Writing $(x',\zeta')=\varphi(y,w)$, we have
\begin{eqnarray*}
(x',\zeta') = \left(\frac{y^2}{w},-w-2\right).
\end{eqnarray*}
There are two systems of coordinates in a neighborhood of $E_1$ within
$\widehat{\mathbb{CP}^2}$. They are $(y, \lambda)$, where $w = \lambda y$, and $(w,\eta)$, where $y=\eta w$.  Then, $\varphi$ lifts to a rational map $\widehat{\varphi}: \widehat{\mathbb{CP}^2} \dashrightarrow\mathbb{CP}^2$ which is expressed in these coordinates by 
\begin{eqnarray}\label{EQN:BLOWUP1}
(x',\zeta') = \left(\frac{y}{\lambda},-\lambda y-2\right) \mbox{  and  } (x',\zeta')=\left(\eta^2 w,-w-2\right).
\end{eqnarray}
The exceptional divisor $E_1$ is given in the first set of coordinates by $y=0$ and in the second set of coordinates by $w=0$.  It is clear from 
(\ref{EQN:BLOWUP1}) that $\hat \varphi$ extends holomorphically to all points of $E_1$ other than $\lambda = 0$, sending each of these points to $(x,\zeta) = (0,-2)$.

In the $(y,w)$ coordinates, $L_{\rm coll}$ is given by $w = 0$.  Thus, in the $(y, \lambda)$ coordinates, the proper transform $\widehat{L}_{\rm coll}$ is given by $\lambda=0$ so that $E_1 \cap \widehat{L}_{\rm coll}$ is given by $(0,0)$, the indeterminate point for $\hat \varphi$ on $E_1$.  We now blow this point up using two new systems of coordinates in a neighborhood of the new exceptional divisor $E_2$.  They are $(y,\tau )$, where $\lambda=\tau y$, and $(\lambda,\theta )$, where $y=\theta \lambda$.

The map $\widehat{\varphi}$ lifts to a map $\widetilde{\varphi}$, which can be expressed in local coordinates as
\begin{eqnarray}\label{EQN:BLOWUP2}
(y',z') = \left({\tau },-\tau ^2 y^2-2 \tau  \right) \mbox{  and  }
(x',\zeta') = \left(\theta ,-\theta \lambda^2  -2\right) 
\end{eqnarray}
This shows that $\widetilde{\varphi}: \widetilde{\mathbb{CP}^2} \dashrightarrow \mathbb{CP}^2$ is holomorphic in a neighborhood of $E_2$.  Thus, it is holomorphic in a neighborhood of $E_1 \cup E_2$.  Since $E_2$ is given in these systems of coordinates by $y=0$ and $\lambda = 0$, respectively,  one can see from (\ref{EQN:BLOWUP2}) that $\widetilde{\varphi}(E_2) = \{Z=-2Y\}$.

Notice that the line $C_1 := \{Z=-2Y\}$ passes through no points of 
\begin{eqnarray*}
\mathcal{I}_{\varphi^3} = \left\{[1:0:1],[0:\pm i: 1],[1:0:-1\pm i]\right\}.
\end{eqnarray*}
This implies that 
\begin{eqnarray*}
\widetilde{{\varphi^4}} := \varphi^3 \circ \widetilde{\varphi} : \widetilde{\mathbb{CP}^2} \dashrightarrow 
\mathbb{CP}^2
\end{eqnarray*}
is holomorphic in a neighborhood of $E_1 \cup E_2$; i.e., that $\widetilde{{\varphi^4}}$ resolves the indeterminacy of ${\varphi^4}$ at $p$.
Since $\widetilde{\varphi}(E_1) = [0:1:-2]$ it is easy to check that $\widetilde{{\varphi^4}}(E_1)  = [1:0:-9] := s$.

Meanwhile, $\widetilde{\varphi}(E_2)$ is the projective line $C_1$. Its forward image under $\varphi^3$ is an algebraic curve $C_4 := \varphi^3(C_1)$ of degree $8 = \deg(\varphi^3)$. Since $C_1$ does not intersect $\mathcal{I}_{\varphi^3}$, $C_4$ is irreducible.

It remains to show that $C_4$ is singular at $s$.  Notice that $\varphi$ maps a neighborhood of $[0:1:3]$ biholomorphically to a neighborhood of $s$.  Thus, it will be sufficient to show that $C_3 := \varphi^2 (C_1)$ is singular at $[0:1:3]$.  See Figure~\ref{CURVES} for plots of the real slices of $C_3$ and $C_4$.

\begin{figure}
\centering {
\scalebox{.75}{
\epsfig{file=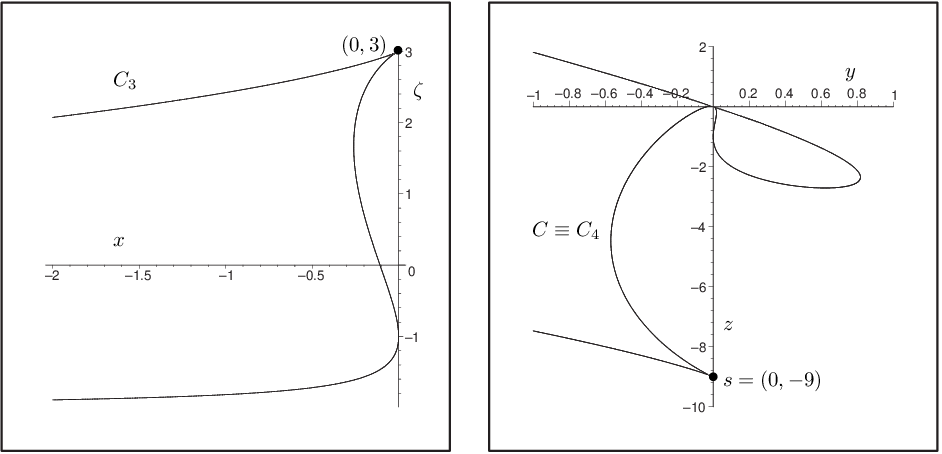}%
}}
\caption{\label{CURVES}The curve $C_3$ shown on the left in the $(x,\zeta)$ coordinates and the curve $C_4$ shown on the right in the $(y,z)$ coordinates.  The cusps at $(x,\zeta) = (0,3)$ and $(y,z) = (0,-9)$, respectively, are marked.}
\end{figure}

Let us work in the $(x,\zeta)$ local coordinates. We have 
\begin{eqnarray*}
(x',\zeta')  = \varphi^2(x,\zeta)  = \left({\frac {{x}^{2} \left( x-\zeta \right) ^{2}}{-1+{x}^{2}-{\zeta}^{2}}} , -1-{x}^{2}+{\zeta}^{2} \right).
\end{eqnarray*}
If we parameterize $C_1$ by $t \mapsto (t,-2)$ and then compose with $\varphi^2$, we obtain the following parameterization of $C_3$:
\begin{eqnarray}\label{PARAM}
t \mapsto \left({\frac {{t}^{2} \left( t+2 \right) ^{2}}{-5+{t}^{2}}},3-{t}^{2}\right).
\end{eqnarray}
To see that $C_3$ is singular at $(0,3)$, it suffices to check that $C_3$ intersects any line through $(0,3)$ with intersection number $\geq 2$.  Such a generic line is given by $A x + B(\zeta-3) = 0$.  Substituting (\ref{PARAM}) into this equation, we obtain 
\begin{eqnarray*}
{\frac {A{t}^{2} \left( t+2 \right) ^{2}}{-5+{t}^{2}}}-B{t}^{2}=0,
\end{eqnarray*}
which clearly has a double zero at $t=0$, implying that the intersection number of $C_3$ with the line given by $A x + B(\zeta-3) = 0$ is at least two.
\end{proof}

\begin{rema}
Applying $\varphi$ to the parameterization of $C_3$ given in~(\ref{PARAM}) one gets a parameterization of~$C_4$.   One can then use Gr\"obner bases \cite[Ch. 3, \S 3]{CLO} to generate an implicit equation for $C_4$ (which we omit here).
Its lowest degree terms are quadratic, which gives an alternative computational proof that $C_4$ is singular at $s$.
\end{rema}

\section{Proof of non-existence of invariant foliations for $\varphi$}\label{SEC:PROOF1}

The whole proof will take place in the coordinates $(y,z)=(Y/X,Z/X)$.  Throughout the proof we use implicitly that $(\varphi^m)^*(\varphi^n)^* \mathcal{F} = (\varphi^{m+n})^* \mathcal{F}$ for $m,n \geq 1$, which follows from Remark~\ref{LEM_COMPOSITION}.

The $\varphi$-preimage of $L_{\rm coll}$ is the circle $S:=\{y^2+z^2=1\}$ which intersects $L_Y$ at the indeterminate point $p=(0,1)$ and at $q_3=(0,-1)$. It follows that $\varphi^2$ is a finite holomorphic map at each point of $L_Y$ other than $p$ and $q_3$, and consequently, ${\varphi^4}$ is a finite holomorphic map at every point of $L_Y$ except at the finite set of points ${\rm NF} := \{ (0,z) \, | \, z=1,-1,0,-2,-1\pm i \}$ consisting of $p$, $q_3$, and their $\varphi^2$-preimages.

One can check that the point $r  := (0,-1+\sqrt{2}i) \not\in {\rm NF}$ is a
${\varphi^4}$-preimage of the singular point $s$ of the blow-up curve $C$.  We will first show for any foliation $\mathcal{F}$ that $({\varphi^4})^{\ast} \mathcal{F}$ is singular either at the indeterminate point $p$ or at $r $.  According to Remark \ref{LEM_LEAF_OR_TRANSV}, either $C$ is a leaf of $\mathcal{F}$ or generically transverse to $\mathcal{F}$.  In the second case, Lemma \ref{CURVE_TRANS_FOLIATION} immediately implies that $({\varphi^4})^{\ast}\mathcal{F}$ is singular at $p$.  On the other hand, if $C$ is a leaf of $\mathcal{F}$, $s$ must be a singular point for $\mathcal{F}$, since $s$ is a singular point of $C$ (Lemma \ref{C_SINGULARITY}). Since ${\varphi^4}$ is a finite holomorphic map at $r $ and ${\varphi^4}(r )=s$, by Lemma \ref{PREIMAGE_SINGULAR}, $({\varphi^4})^{\ast}\mathcal{F}$ is singular at $r $ as well.

Note that neither $p$ nor $r $ are critical points of $\varphi_{|L_Y}$, so they are not exceptional points.  One can check that any preorbit of $p$ or $r $ under ${\varphi^4}_{|L_Y}$ is disjoint from ${\rm NF}$.   Let $a _0$ be the point ($p$~or~$r $) where $({\varphi^4})^{\ast}\mathcal{F}$ is singular.  If we denote some ${\varphi^4}_{|L_Y}$-preorbit of $a _0$ by $\{ a _{-i} \}_{i=0}^{\infty}$, then for each $i$, $\big( \varphi^{4(i+1)} \big)^{\ast}\mathcal{F}$ will be singular at $a _{-i}$.

Now suppose the foliation $\mathcal{F}$ is preserved by $\varphi^{\ell}$ for any positive integer $\ell$. Then $\mathcal{F} = \big( \varphi^{4\ell k} \big)^* \mathcal{F}$ is singular at $a _{(-\ell k)}$ for all $k \geq 1$.  This implies $\mathcal{F}$ has infinitely many singular points, giving a contradiction.  
\hfill \qed

\bigskip\noindent{\bf Observation:} It was more convenient to explain the proof with a specific map; however the proof holds in greater generality. In fact, we have established

\begin{theo}\label{THM:CONDITIONS}
  Assume the map $\eta:\mathbb{CP}^2 \dashrightarrow \mathbb{CP}^2$ satisfies the following conditions:
  \begin{enumerate}
    \item There exists $p \in \mathcal{I}_\eta$ and some iterate $k$ so that $\eta^k$ blows up $p$ to a singular curve $C$.
    \item The point $p$ has an infinite preorbit along which $\eta$ is a finite holomorphic map.
    \item There is a singular point $s \in C$ having an infinite preorbit along which $\eta$ is a finite holomorphic map.
  \end{enumerate}
  Then no iterate of $\eta$ preserves a foliation.
\end{theo}

Note that one can replace Conditions 2 and 3 with the equivalent condition that $p$ and $s$ have infinite $\eta^k$-preorbits along which $\eta^k$ is a finite holomorphic map.

\begin{rema}
A theorem of Jouanolou \cite{JOU,GHYS} asserts that a foliation $\mathcal{F}$
cannot have algebraic leaves of arbitrarily high degree.  This implies the
following variant of Theorem \ref{THM:CONDITIONS}:
\theoremstyle{theo}
\newtheorem*{theop}{Theorem 4.1'}
\begin{theop}
Assume the map $\eta:\mathbb{CP}^2 \dashrightarrow \mathbb{CP}^2$ has an indeterminate point $p$ such that
  \begin{enumerate}
    \item For each $n$ the blow-up $\eta^n(p)$ contains an irreducible curve $C_n$ so that $\limsup \, \deg(C_n) = \infty$.
    \item The point $p$ has an infinite preorbit along which $\eta$ is a finite holomorphic map.
  \end{enumerate}
  Then no iterate of $\eta$ preserves a foliation.
\end{theop}
Notice that Condition 1 holds if the orbit of $\eta(p)$ avoids $\mathcal{I}_\phi$, and each curve maps to the next
with topological degree less than $\deg_{\rm alg}(\eta)$.  Also, since the curves are rational, the degree-genus formula implies that they will develop singularities.  However in certain applications it may be hard to verify 
that these singularities satisfy Condition 3 in Theorem \ref{THM:CONDITIONS}.  In that situation Theorem 4.1' can be useful.

\end{rema}

\section{Proof of non-existence of invariant foliations for $\psi$}\label{SEC:PROOF2}

The extension of $\psi(x,y) = \big(x(x-y)+2,(x+y)(x-y)+1\big)$ to $\mathbb{CP}^2$ is 
\begin{eqnarray}
\psi[X:Y:Z] = [X(X-Y)+2Z^2:(X+Y)(X-Y)+Z^2:Z^2].
\end{eqnarray}
Let us first verify that $\psi$ is algebraically stable.  The indeterminate set consists of the point
\begin{eqnarray*}
p := [1:1:0],
\end{eqnarray*}
and the critical set is
\begin{eqnarray*}
\{X=Y\} \cup \{Z = 0\}.
\end{eqnarray*}
One can check that $L_{\rm coll}:=\{X=Y\}$ satisfies that $\psi(L_{\rm coll} \setminus \{p\}) = [2:1:1]$, which is periodic of period $2$:
\[[2:1:1] \leftrightarrow [4:4:1].\]
The line $L_{\infty}:= \{Z=0\}$ is not collapsed by $\psi$; using the coordinate $w = Y/X$, the restriction $\psi_{\mid L_\infty}$ is given by $w \mapsto w+1$.  Since no iterate of a collapsing line lands on the indeterminate point $p$, it follows from \cite[Prop.1.4.3]{SIBONY} that $\psi$ is algebraically stable; thus, $\lambda_1(\psi) = 2$.

Any point $(a,b) \in \mathbb{C}^2$ with $2\,a-b-3 \neq 0$ has two preimages in $\mathbb{C}^2$ given by
\begin{eqnarray}\label{EQN:PREIMAGES}
(x,y) = \left(\pm{\frac {-2+a}{\sqrt {2\,a-b-3}}},\mp {\frac {a-1-b}{\sqrt {2\,a-b-3}}}\right).
\end{eqnarray}
Since generic points of $\mathbb{C}^2$ do not have a preimage on $L_\infty$, $\lambda_2(\psi) = 2$; thus $\psi$ is not cohomologically hyperbolic.

\begin{proof}[Proof of Theorem \ref{THM2}]

It suffices to verify that $\psi$ satisfies the conditions of Theorem
\ref{THM:CONDITIONS}.  

A sequence of two blow-ups at $p$ resolves the indeterminacy of $\psi$.  The first blow-up results in a lift $\widehat \psi: \widehat{\mathbb{CP}^2} \dashrightarrow \mathbb{CP}^2$ with exceptional divisor $E_1$.  The point $q \in E_1$ with coordinate $\lambda := (Y-X)/Z = 0$ is indeterminate for $\widehat \psi$, and $\widehat \psi\left(E_1 \setminus \{q\} \right) = [1:2:0]$.  The second blow-up (at $q$) resolves the indeterminacy, and the
lifted map $\widetilde{\psi}$ sends the new exceptional divisor $E_2$ to the line $C_1:= \{2X-Y-3Z=0\}$.  We omit the calculations as they are very similar to those in Lemma~\ref{C_SINGULARITY}.

The indeterminacy set of $\psi^3$ is
\begin{eqnarray*}
\mathcal{I}_{\psi^3} = \mathcal{I}_{\psi} \cup \psi^{-1}(\mathcal{I}_\psi) \cup \psi^{-2}(I_\psi) = \{[1:1:0]\} \cup \{[1:0:0]\} \cup \{[1:-1:0]\}.
\end{eqnarray*}
Since $\mathcal{I}_{\psi^3}$ is disjoint from $C_1$, the map $\widetilde{\psi^4} := \psi^3 \circ \widetilde{\psi}$ resolves the indeterminacy of $\psi^4$ at $p$.  Moreover, $\widetilde{\psi^4}$ blows up $p$ to the curve
\begin{eqnarray*}
C_4:=\widetilde{\psi^4}(E_2) = \psi^3(C_1).
\end{eqnarray*}
We will now check that $s = [4:4:1]$ is a singular point on $C_4$.  Let us work in the affine coordinates $(x,y) = (X/Z,Y/Z)$.  If we parameterize $C_1$ by $t \mapsto (t,2t-3)$, then a parameterization of $C_3: = \psi^2(C_1)$ is obtained by substituting into the expression for $\psi^2$:
\begin{align}\label{EQN:C2PARAM}
 t \mapsto
\left(-2\,{t}^{4}+15\,{t}^{3}-33\,{t}^{2}+12\,t+22 \, , \, -8\,{t}^{4}+66\,{t}^{3}-187\,{t}^{2}+204\,t-59\right).
\end{align}

One can see that $C_3$ is singular at $(2,1)$ by using Gr\"obner bases to convert (\ref{EQN:C2PARAM}) into the following implicit equation for $C_3$
\begin{eqnarray*}
256\,{u}^{4}-256\,{u}^{3}v+96\,{u}^{2}{v}^{2}-16\,u{v}^{3}+{v}^{4}-5650\,{u}^
{3}+6253\,{u}^{2}v && \\-2228\,u{v}^{2}+257\,{v}^{3} +10816\,{u}^{2}-10816\,uv+2704
\,{v}^{2} &=& 0,
\end{eqnarray*}
which is expressed in local coordinates $(u,v)$ where $x=u+2, y = v+1$.
Since the lowest order terms are of degree $2$, the point $(x,y) = (2,1)$ is singular for $C_3$. (This can also be shown using the parameterization of $C_3$, like in the proof of Lemma~\ref{C_SINGULARITY}.)

One can check that $D\psi$ is invertible at $(2,1)$ so that $C_4 = \psi(C_3)$
is singular at $s = \psi(2,1) = (4,4)$.  Therefore, Condition (1) from Theorem \ref{THM:CONDITIONS} holds.

Since the action of $\psi$ on $L_\infty$ is given by $w \mapsto w+1$, and $p$ is given in this coordinate by $w = 1$, we find that $p$ has an infinite pre-orbit under $\psi$.  Moreover, these preimages are disjoint from $p$ (and hence from the collapsing line $L_{\rm coll}$), so Condition 2 of the theorem holds.

It remains to check Condition (3).  Notice that $\psi(L_{\rm coll} \setminus \{p\}) = (2,1) \in C_1$.  Therefore, it suffices to show that $(4,4) \not \in C_1$ has an infinite pre-orbit in $\mathbb{C}^2$ consisting of points not on $C_1$. For any $(a,b) \not \in C_1$, the two preimages given by
(\ref{EQN:PREIMAGES}) cannot both be on $C_1$, since in that case
\begin{eqnarray*}
\frac{-5+3a-b+3\sqrt{2a-b-3}}{\sqrt{2a-b-3}} &=& 0, \mbox{and} \\
\frac{5-3a+b+3\sqrt{2a-b-3}}{\sqrt{2a-b-3}} &=& 0.
\end{eqnarray*}
Summing these equations yields $-6 = 0$, a contradiction. Therefore, any point of $\mathbb{C}^2 \setminus C_1$ has an infinite pre-orbit disjoint from 
$C_1$.

Since all three conditions of Theorem \ref{THM:CONDITIONS} hold, we conclude that no iterate of $\psi$ preserves a foliation.
\end{proof}

\section{Proof of non-existence of invariant foliations for generic rotations of ${\varphi^4}$}\label{SEC:PROOF3}

Note that for any $A \in {\rm PGL}(3,\mathbb{C})$, we have $\mathcal{I}_{A \circ {\varphi^4}} = \mathcal{I}_{\varphi^4}$. For any $n \in \mathbb{Z}^+$, let
{\small{\begin{eqnarray*}
\Omega^n_1 &:=& \big\{A \in {\rm PGL}(3,\mathbb{C}) \, : \, (A \circ {\varphi^4})^n(L_{\rm coll} \setminus \mathcal{I}_{\varphi^4})  \not \in \{p,r\}\big\}, \\
\Omega^n_2 &:=& \big\{A \in {\rm PGL}(3,\mathbb{C}) \, : \, p,q \not \in (A \circ {\varphi^4})^n(\mathcal{I}_{\varphi^4}) \big\}, \, \mbox{and} \\
\Omega^n_3 &:=& \big\{A \in {\rm PGL}(3,\mathbb{C}) \, : \, \mbox{ for all } 0 \leq i \neq j \leq n, \, (A \circ {\varphi^4})^{-i}(p) \cap (A \circ {\varphi^4})^{-j}(p) = \emptyset \big\} \cap \\ && \left\{A \in {\rm PGL}(3,\mathbb{C}) \, : \, \mbox{ for all } 0 \leq i \neq j \leq n, \, (A \circ {\varphi^4})^{-i}(r) \cap (A \circ {\varphi^4})^{-j}(r) = \emptyset\right\}. \nonumber
\end{eqnarray*}}}
Here, $({\varphi^4})^n(\mathcal{I}_{\varphi^4})$ denotes the total transform of $\mathcal{I}_{\varphi^4}$, i.e. $({\varphi^4})^n(\mathcal{I}_{\varphi^4})= ({\varphi^4})^{n-1}(C_1)$, where ${\varphi^4}$ blows up $\mathcal{I}_{\varphi^4}$ to~$C_1$.

Each of these sets is the complement of an algebraic set because the conditions
\begin{enumerate}
  \item \label{one}
        $(A \circ {\varphi^4})^n$ collapses $L_{\rm coll}$ on $p$ or in $r$,
  \item \label{two}
        $p,r \in (A \circ {\varphi^4})^n\mathcal{I}_{\varphi^4}$,
  \item \label{three}
        $(A \circ {\varphi^4})^{-i}(p) \cap (A \circ {\varphi^4})^{-j}(p) \neq \emptyset$, and
  \item \label{four}
        $(A \circ {\varphi^4})^{-i}(r) \cap (A \circ {\varphi^4})^{-j}(r) \neq \emptyset$
\end{enumerate}
are all algebraic. In order to conclude that $\Omega^n_1 \cap \Omega^n_2 \cap \Omega^n_3$ is a dense open subset of ${\rm PGL}(3,\mathbb{C})$ we only need to show that this intersection is not empty. Notice that ${\rm id} \in \Omega^n_1 \cap \Omega^n_2$ follows from our study of ${\varphi^4}$. For even $i,j$, Property~\ref{three} holds for $A = {\rm id}$ because $p$ is in the basin of $\infty$ for ${\varphi^4}_{|L_Y}$. Similarly, for odd $i,j$ by taking ${\varphi^4}(p)$, which is in the basin of $\infty$ for ${\varphi^4}_{|L_X}$.

By Baire's Theorem the intersection $\Omega^\infty:=\bigcap_n \Omega^n_1 \cap \Omega^n_2 \cap \Omega^n_3$ is generic in ${\rm PGL}(3,\mathbb{C})$.

For $A \in \Omega^\infty$, $A \circ {\varphi^4}$ is algebraically stable, since $A \not \in \Omega^n_1$ for any $n$.  Thus $\lambda_1(A \circ {\varphi^4}) = 16$. Meanwhile, for any $A \in {\rm PGL}(3,\mathbb{C})$, $\lambda_2(A \circ {\varphi^4}) = \lambda_2({\varphi^4}) = 16$. Thus, for any $A \in \Omega^\infty$, $A \circ {\varphi^4}$ is non-cohomologically hyperbolic.

We will now verify that for $A \in \Omega^\infty$ the composition $A \circ
{\varphi^4}$ satisfies the hypotheses of Theorem~\ref{THM:CONDITIONS}.
Condition (1) follows with $k=1$ from Lemma~\ref{C_SINGULARITY} since
${\varphi^4}$ blows-up $p$ to $C$, which is singular at $s$.  Thus, $A \circ
{\varphi^4}$ blows-up $p$ to $A(C)$, which is singular at $A(s)$.  Since $A 
\in \Omega_1^n \cup \Omega_2^n \cup \Omega_3^n$ for any $n$, the point $p$ has an infinite pre-orbit under $A \circ {\varphi^4}$ along which $A \circ {\varphi^4}$ is finite. Thus, Condition (2) holds.  Finally, since ${\varphi^4}$ is finite at $r$ with ${\varphi^4}(r) = s$, the composition $A \circ {\varphi^4}$ is finite at $r$ and maps $r$ to $A(s)$.  Condition (3) then follows from the fact that $A \in \Omega_1^n \cup \Omega_2^n \cup \Omega_3^n$.

Since the hypotheses of Theorem~\ref{THM:CONDITIONS} are satisfied, no iterate of $A \circ {\varphi^4}$ preserves a foliation. \qed \\

\noindent{\bf Observation:} All the arguments above hold with $\varphi$ substituting ${\varphi^4}$, except for the verification of Condition~(3), because it is not clear how to ensure that $(A \circ \varphi)^4(p)$ remains singular and has an infinite pre-orbit along which $A \circ \varphi$ is finite. For the sake of simplicity we decided not to address this technicality, but we expect that for a generic rotation of $\varphi$ the same result holds.

\begin{question}
(Ch.~Favre) Let $\eta: \mathbb{CP}^2 \dashrightarrow \mathbb{CP}^2$ be an algebraically stable map that is not cohomologically hyperbolic.  Then generic rotations $A\circ\eta$ will have the same dynamical degrees and hence also be non-cohomologically hyperbolic.  Under what conditions on $\eta$ will generic rotations of $\eta$ not preserve a foliation?
\end{question}

\def\cprime{$'$}

\end{document}